\documentclass{amsart}
\usepackage{amssymb}
\usepackage{amsfonts}

\newtheorem{theorem}{Theorem}
\theoremstyle{plain}

\newtheorem{corollary}{Corollary}

\newtheorem{lemma}{Lemma}

\newtheorem{proposition}{Proposition}
\newtheorem{remark}{Remark}

\numberwithin{equation}{section}
\input{tcilatex}

\begin{document}
\title[On Iyengar-Type Inequalities]{On Iyengar-Type Inequalities via
Quasi-Convexity and Quasi-Concavity{\large \textbf{\ }}}
\author{M. Emin \"{O}zdemir$^{\bigstar }$}
\address{$^{\bigstar }$ATAT\"{U}RK UNIVERSITY, K.K. EDUCATION FACULTY,
DEPARTMENT OF MATHEMATICS, 25240, CAMPUS, ERZURUM, TURKEY}
\email{emos@atauni.edu.tr}
\thanks{$^{\bigstar }$Corresponding Author}
\subjclass[2000]{Primary 26D15.}
\keywords{Weighted H\"{o}lder Inequality, H\"{o}lder Inequality, Power-mean
Inequality, Differentiable Function, quasi-convex Function.}

\begin{abstract}
In this paper, we obtain some new estimations of Iyengar-type inequality in
which quasi-convex(quasi-concave) functions are involved. These estimations
are improvements of some recently obtained estimations. Some error
estimations for the trapezoidal formula are given. Applications for special
means are also provided.
\end{abstract}

\maketitle

\section{Short Historical Background and Introduction}

If it is necessary to bound one quantity by another, the classical
inequalities are very useful for this purpose. This first book called "
Inequalities " written by Hardy, Littlewood and Polya at cambridge
University Press in 1934 represents the first effort to systemize a rapidly
expanding domain. In this sense, the second important book " Classical and
New Inequalities in Analysis " is written by D.S. Mitrinovi\'{c}, J.E.Pe\'{c}%
ari\v{c} and A.M. Fink. The third book " Analytic Inequalities " written by
D.S. Mitrinovi\'{c}, and the other book " Means and Their Inequalities "
written by Bullen, D.S. Mitrinovi\'{c}, D.S. Vasic, P.M.

Today inequalities play a significant role for the development in all fields
of Mathematics. They have applications in a variety of applied Mathematics.
\ For example, convex functions are tractable in optimization because local
optimality guarantees global optimality. In recent years a number of authors
have discovered new integral inequalities for convex, $s-$convex functions,
logarithmic convex functions, $h-$convex functions, $quasi$-convex
functions, $m-$convex functions, $(\alpha ,m)-$convex functions,
co-ordinated convex functions, and Godunova-Levin function, $P-$function.

On November 22, 1881, Hermite (1822-1901) sent a letter to the Journal
Mathesis. This letter was published in Mathesis 3 (1883,p.82) and in this
letter an inequality presented which is well-known in the literature as
Hermite-Hadamard integral inequality : 
\begin{equation}
f\left( \frac{a+b}{2}\right) \leq \frac{1}{b-a}\int_{a}^{b}f\left( x\right)
dx\leq \frac{f\left( a\right) +f\left( b\right) }{2},  \label{1.1}
\end{equation}%
where $f:I\subseteq 
\mathbb{R}
\rightarrow 
\mathbb{R}
$\ is a convex function on the interval $I$\ of a real numbers and $a,b\in
I\;$with $a<b.$ If the function $f\;$is concave, the inequality in (\ref{1.1}%
) is reversed. That is 
\begin{equation*}
f\left( \frac{a+b}{2}\right) \geq \frac{1}{b-a}\int_{a}^{b}f\left( x\right)
dx\geq \frac{f\left( a\right) +f\left( b\right) }{2}.
\end{equation*}%
For recent results, generalizations and new inequalities related to the
inequality (\ref{1.1}) see (\cite{7}-\cite{14.})$.$

Then left hand side of Hermite-Hadamard inequality $(LHH)\;$can also be
estimated by the inequality of Iyengar.%
\begin{equation}
\frac{f\left( a\right) +f\left( b\right) }{2}-\frac{1}{b-a}%
\int_{a}^{b}f\left( x\right) dx\leq \frac{M\left( b-a\right) }{4}-\frac{%
\left[ f\left( b\right) -f\left( a\right) \right] ^{2}}{4M\left( b-a\right) }
\label{1.2}
\end{equation}%
where 
\begin{equation*}
M=\sup \left\{ \left\vert \frac{f\left( x\right) -f\left( y\right) }{x-y}%
\right\vert ;x\neq y\right\}
\end{equation*}%
In \cite{3}, Daniel Alexandru Ion proved the following inequalities of
Iyengar type for differentiable $quasi-$convex functions:\qquad 
\begin{equation}
\left\vert \frac{f\left( a\right) +f\left( b\right) }{2}-\frac{1}{b-a}%
\int_{a}^{b}f\left( x\right) dx\right\vert \leq \frac{\left( b-a\right) }{4}%
\left( \sup \left\{ \left\vert f^{\prime }\left( a\right) \right\vert
,\left\vert f^{\prime }\left( b\right) \right\vert \right\} \right)
\label{1.3}
\end{equation}%
where $f:\left[ a,b\right] \rightarrow 
\mathbb{R}
\;$is differentiable function on $\left( a,b\right) ,\;$and $\left\vert
f^{\prime }\right\vert \;$is $quasi-$convex on $\left[ a,b\right] \ $with $%
a<b.$

and 
\begin{eqnarray}
&&\left\vert \frac{f\left( a\right) +f\left( b\right) }{2}-\frac{1}{b-a}%
\int_{a}^{b}f\left( x\right) dx\right\vert  \label{1.4} \\
&\leq &\frac{\left( b-a\right) }{2\left( p+1\right) ^{\frac{1}{p}}}\left(
\sup \left\{ \left\vert f^{\prime }\left( a\right) \right\vert ^{\frac{p}{p-1%
}},\left\vert f^{\prime }\left( b\right) \right\vert ^{\frac{p}{p-1}%
}\right\} \right) ^{\frac{p-1}{p}}  \notag
\end{eqnarray}

where $f:\left[ a,b\right] \rightarrow 
\mathbb{R}
\;$is differentiable function on $\left( a,b\right) ,\;$and $\left\vert
f^{\prime }\right\vert ^{\frac{p}{p-1}}\;$is $quasi-$convex on $\left[ a,b%
\right] \;$with $a<b.$

We give some necessary definitions and mathematical preliminaries for $%
quasi- $convex functions which are used throughout this paper.

\textbf{Definition 1. (}see \cite{1}) A function $f:\left[ a,b\right]
\rightarrow 
\mathbb{R}
\;$is said to be $quasi-$convex on $\left[ a,b\right] \;$if 
\begin{equation}
f\left( \lambda x+\left( 1-\lambda \right) y\right) \leq \max \left\{
f\left( x\right) ,f\left( y\right) \right\} ,  \label{(QC)}
\end{equation}%
holds for all $x,y\in \left[ a,b\right] \;$and $\lambda \in \left[ 0,1\right]
.$For additional results on $quasi-$convexity, see \cite{2}. Clearly, any
convex function is $quasi-$convex function. Furthermore, there exists $%
quasi- $convex functions which are not convex. See \cite{3} : 
\begin{equation*}
g\left( t\right) =\left\{ 
\begin{array}{ccc}
1, &  & t\in \left[ -2,-1\right] \\ 
&  &  \\ 
t^{2}, &  & t\in \left( -1,2\right]%
\end{array}%
\right.
\end{equation*}%
is not a convex function on $\left[ -2,2\right] ,$but it is a $quasi-$convex
function on $\left[ -2,2\right] .$If we choose $g:\left[ -2,2\right]
\rightarrow 
\mathbb{R}
,$ $g\left( -2\right) =1,$ $g\left( 2\right) =4\;$and for $\alpha =\frac{1}{2%
},$ $a=-2,$ $b=0,$ we get $g\left( \alpha a+\left( 1-\alpha \right) b\right)
=g\left( -1\right) =1\;$and $\alpha g\left( a\right) +\left( 1-\alpha
\right) g\left( b\right) =\frac{1}{2}g\left( -2\right) +\frac{1}{2}g\left(
0\right) =\frac{1}{2}.$Thus it is not convex but it is $quasi-$convex
function for all $\alpha \in \left[ 0,1\right] ,g\left( -\alpha 2+\left(
1-\alpha \right) 2\right) \leq \max \left\{ g\left( -2\right) ,g\left(
2\right) \right\} =\max \left\{ 1,4\right\} =4.$

The main purpose of this paper is to point out new estimations of the
inequality in (\ref{1.2}) , but now for the class of $quasi-$convex
functions.

In order to prove our main results we need the following lemma (see \cite{4})%
$.$

\begin{lemma}
\label{lem 1.1} Let $\ f:I\subset 
\mathbb{R}
\rightarrow 
\mathbb{R}
$ be a twice differentiable mapping on $I^{\circ },$ $a,b\in I$ with $a<b$
and $f^{\prime \prime }$ be integrable on $[a,b].$ Then the following
equality holds:%
\begin{equation*}
\frac{f(a)+f(b)}{2}-\frac{1}{b-a}\int_{a}^{b}f(x)dx=\frac{\left( b-a\right)
^{2}}{2}\int_{0}^{1}t\left( 1-t\right) f^{\prime \prime }\left(
ta+(1-t)b\right) dt.
\end{equation*}
\end{lemma}

The main results of this paper are given by the following theorems.

\section{The Results}

\begin{theorem}
\label{teo 2.1} Let $f:I^{\circ }\subset \lbrack 0,\infty )\rightarrow 
\mathbb{R}
,$ be a twice differentiable mapping on $I^{\circ },$ such that $f^{\prime
\prime }\in L[a,b],$ $a,b\in I$ with $a<b$. If $\left\vert f^{\prime \prime
}\right\vert ^{q}$ is $quasi-$convex on $[a,b]$ for $q>1,$ then the
following inequality holds:%
\begin{eqnarray}
&&\left\vert \frac{f(a)+f(b)}{2}-\frac{1}{b-a}\int_{a}^{b}f(x)dx\right\vert
\label{2.1} \\
&\leq &\frac{\left( b-a\right) ^{2}}{2}\left( \frac{q-1}{2q-p-1}\right) ^{%
\frac{q-1}{q}}\left( \beta \left( p+1,q+1\right) \right) ^{\frac{1}{q}} 
\notag \\
&&\times \left( \max \left\{ \left\vert f^{\prime \prime }(a)\right\vert
^{q},\left\vert f^{\prime \prime }(b)\right\vert ^{q}\right\} \right) ^{%
\frac{1}{q}}  \notag
\end{eqnarray}%
where $\frac{1}{p}+\frac{1}{q}=1$ and $\beta \left( \text{ },\right) $ is
Euler Beta Function:%
\begin{equation*}
\beta \left( x\text{ },y\right) =\int_{0}^{1}t^{x-1}\left( 1-t\right)
^{y-1}dt,\text{ \ \ \ }x,y>0.
\end{equation*}
\end{theorem}

\begin{theorem}
\label{teo 2.2} Let $f:I^{\circ }\subset \lbrack 0,\infty )\rightarrow 
\mathbb{R}
,$ be a twice differentiable mapping on $I^{\circ },$ such that $f^{\prime
\prime }\in L[a,b],$ $a,b\in I$ with $a<b$. If $\left\vert f^{\prime \prime
}\right\vert ^{q}$ is $quasi-$convex on $[a,b]$ for $q\geq 1,$ then the
following inequality holds:%
\begin{eqnarray}
&&\left\vert \frac{f(a)+f(b)}{2}-\frac{1}{b-a}\int_{a}^{b}f(x)dx\right\vert
\label{2.2} \\
&\leq &\frac{\left( b-a\right) ^{2}}{4}\left( \frac{2}{\left( q+1\right)
\left( q+2\right) }\right) ^{\frac{q-1}{q}}\left( \max \left\{ \left\vert
f^{\prime \prime }(a)\right\vert ^{q},\left\vert f^{\prime \prime
}(b)\right\vert ^{q}\right\} \right) ^{\frac{1}{q}}.  \notag
\end{eqnarray}
\end{theorem}

\begin{theorem}
\label{teo 2.3} With the assumptions of Theorem \ref{teo 2.1}, we obtain
another 
\begin{eqnarray*}
&&\left\vert \frac{f(a)+f(b)}{2}-\frac{1}{b-a}\int_{a}^{b}f(x)dx\right\vert
\\
&\leq &\frac{\left( b-a\right) ^{2}}{2^{1+\frac{1}{q}}}\left( \beta \left(
2,p+1\right) \right) ^{\frac{1}{p}}\left( \max \left\{ \left\vert f^{\prime
\prime }(a)\right\vert ^{q},\left\vert f^{\prime \prime }(b)\right\vert
^{q}\right\} \right) ^{\frac{1}{q}}.
\end{eqnarray*}
\end{theorem}

\section{proof of main results}

\textbf{Proof of Theorem \ref{teo 2.1}: }Using Lemma \ref{lem 1.1} and the
well known H\"{o}lder's inequality for $q>1,$ 
\begin{eqnarray*}
&&\left\vert \frac{f\left( a\right) +f\left( b\right) }{2}-\frac{1}{b-a}%
\int_{a}^{b}f\left( x\right) dx\right\vert \\
&\leq &\frac{\left( b-a\right) ^{2}}{2}\left( \int_{0}^{1}t^{\frac{q-p}{q-1}%
}dt\right) ^{\frac{q-1}{q}}\left[ \int_{0}^{1}t^{p}\left( 1-t\right)
^{q}\left\vert f^{\prime \prime }\left( ta+\left( 1-t\right) b\right)
\right\vert ^{q}dt\right] ^{\frac{1}{q}},
\end{eqnarray*}%
where $\frac{1}{p}+\frac{1}{q}=1.$

On the other hand, since $\left\vert f^{\prime \prime }\right\vert ^{q}$ is $%
quasi-$convex on $[a,b],$ we know that for any $t\in \lbrack 0,1]$%
\begin{equation*}
\left\vert f^{\prime \prime }\left( ta+\left( 1-t\right) b\right)
\right\vert ^{q}\leq \max \left\{ \left\vert f^{\prime \prime
}(a)\right\vert ^{q},\left\vert f^{\prime \prime }(b)\right\vert
^{q}\right\} .
\end{equation*}%
Therefore, we obtain%
\begin{eqnarray*}
&&\left\vert \frac{f\left( a\right) +f\left( b\right) }{2}-\frac{1}{b-a}%
\int_{a}^{b}f\left( x\right) dx\right\vert \\
&\leq &\frac{\left( b-a\right) ^{2}}{2}\left( \int_{0}^{1}t^{\frac{q-p}{q-1}%
}dt\right) ^{\frac{q-1}{q}}\left[ \int_{0}^{1}t^{p}\left( 1-t\right)
^{q}\left\vert f^{\prime \prime }\left( ta+\left( 1-t\right) b\right)
\right\vert ^{q}dt\right] ^{\frac{1}{q}} \\
&\leq &\frac{\left( b-a\right) ^{2}}{2}\left( \int_{0}^{1}t^{\frac{q-p}{q-1}%
}dt\right) ^{\frac{q-1}{q}}\left[ \int_{0}^{1}t^{p}\left( 1-t\right)
^{q}\left( \max \left\{ \left\vert f^{\prime \prime }(a)\right\vert
^{q},\left\vert f^{\prime \prime }(b)\right\vert ^{q}\right\} \right) dt%
\right] ^{\frac{1}{q}} \\
&=&\frac{\left( b-a\right) ^{2}}{2}\left( \frac{q-1}{2q-p-1}\right) ^{\frac{%
q-1}{q}}\left( \beta \left( p+1,q+1\right) \right) ^{\frac{1}{q}}\left( \max
\left\{ \left\vert f^{\prime \prime }(a)\right\vert ^{q},\left\vert
f^{\prime \prime }(b)\right\vert ^{q}\right\} \right) ^{\frac{1}{q}},
\end{eqnarray*}%
which completes the proof.

\begin{corollary}
\label{co 2.1} In Theorem \ref{teo 2.1}, if we choose $M=Sup_{x\in \left(
a,b\right) }\left\vert f^{\prime \prime }(x)\right\vert <\infty ,$ we get%
\begin{eqnarray*}
&&\left\vert \frac{f\left( a\right) +f\left( b\right) }{2}-\frac{1}{b-a}%
\int_{a}^{b}f\left( x\right) dx\right\vert \\
&\leq &\frac{\left( b-a\right) ^{2}}{2}M\left( \frac{q-1}{2q-p-1}\right) ^{%
\frac{q-1}{q}}\left( \beta \left( p+1,q+1\right) \right) ^{\frac{1}{q}}.
\end{eqnarray*}
\end{corollary}

\textbf{Proof of Theorem \ref{teo 2.2}: }From Lemma \ref{lem 1.1} and the
well known power-mean inequality we obtain%
\begin{eqnarray*}
&&\left\vert \frac{f\left( a\right) +f\left( b\right) }{2}-\frac{1}{b-a}%
\int_{a}^{b}f\left( x\right) dx\right\vert \\
&\leq &\frac{\left( b-a\right) ^{2}}{2}\int_{0}^{1}t\left( 1-t\right)
\left\vert f^{\prime \prime }\left( ta+\left( 1-t\right) b\right)
\right\vert dt \\
&\leq &\frac{\left( b-a\right) ^{2}}{2}\left( \int_{0}^{1}tdt\right) ^{1-%
\frac{1}{q}}\left( \int_{0}^{1}t\left( 1-t\right) ^{q}\left\vert f^{\prime
\prime }\left( ta+\left( 1-t\right) b\right) \right\vert ^{q}dt\right) ^{%
\frac{1}{q}} \\
&\leq &\frac{\left( b-a\right) ^{2}}{2}\left( \int_{0}^{1}tdt\right) ^{1-%
\frac{1}{q}}\left( \int_{0}^{1}t\left( 1-t\right) ^{q}\left( \max \left\{
\left\vert f^{\prime \prime }(a)\right\vert ^{q},\left\vert f^{\prime \prime
}(b)\right\vert ^{q}\right\} \right) dt\right) ^{\frac{1}{q}} \\
&=&\frac{\left( b-a\right) ^{2}}{2}\left( \frac{1}{2}\right) ^{1-\frac{1}{q}%
}\left( \frac{1}{(q+1)(q+2)}\right) ^{\frac{1}{q}}\left( \max \left\{
\left\vert f^{\prime \prime }(a)\right\vert ^{q},\left\vert f^{\prime \prime
}(b)\right\vert ^{q}\right\} \right) ^{\frac{1}{q}} \\
&=&\frac{\left( b-a\right) ^{2}}{4}\left( \frac{2}{(q+1)(q+2)}\right) ^{%
\frac{1}{q}}\left( \max \left\{ \left\vert f^{\prime \prime }(a)\right\vert
^{q},\left\vert f^{\prime \prime }(b)\right\vert ^{q}\right\} \right) ^{%
\frac{1}{q}}.
\end{eqnarray*}%
The proof of Theorem \ref{teo 2.2} is completed.

\begin{corollary}
\label{co 2.2} Under the assumptions of Theorem \ref{teo 2.2},
\end{corollary}

\textbf{Case i: }Since $\lim_{q\rightarrow \infty }\left( \frac{2}{(q+1)(q+2)%
}\right) ^{\frac{1}{q}}=1$ and $\lim_{q\rightarrow 1^{+}}\left( \frac{2}{%
(q+1)(q+2)}\right) ^{\frac{1}{q}}=\frac{1}{3},$ we have%
\begin{equation*}
\frac{1}{3}<\left( \frac{2}{(q+1)(q+2)}\right) ^{\frac{1}{q}}<1,\text{ \ \ \
\ }q\in \lbrack 1,\infty ).
\end{equation*}%
Therefore, 
\begin{equation}
\left\vert \frac{f\left( a\right) +f\left( b\right) }{2}-\frac{1}{b-a}%
\int_{a}^{b}f\left( x\right) dx\right\vert \leq \frac{\left( b-a\right) ^{2}%
}{4}\left( \max \left\{ \left\vert f^{\prime \prime }(a)\right\vert
^{q},\left\vert f^{\prime \prime }(b)\right\vert ^{q}\right\} \right) ^{%
\frac{1}{q}}.  \label{3.1}
\end{equation}%
\textbf{\ }In (\ref{3.1}),

\begin{itemize}
\item if $\left\vert f^{\prime \prime }\right\vert ^{q}$ is decreasing, we
get%
\begin{equation*}
\left\vert \frac{f\left( a\right) +f\left( b\right) }{2}-\frac{1}{b-a}%
\int_{a}^{b}f\left( x\right) dx\right\vert \leq \frac{\left( b-a\right) ^{2}%
}{4}\left\vert f^{\prime \prime }(a)\right\vert ,
\end{equation*}

\item if $\left\vert f^{\prime \prime }\right\vert ^{q}$ is increasing , we
get%
\begin{equation*}
\left\vert \frac{f\left( a\right) +f\left( b\right) }{2}-\frac{1}{b-a}%
\int_{a}^{b}f\left( x\right) dx\right\vert \leq \frac{\left( b-a\right) ^{2}%
}{4}\left\vert f^{\prime \prime }(b)\right\vert .
\end{equation*}
\end{itemize}

\textbf{Case ii: }If we choose $M=Sup_{x\in (a,b)}\left\vert f^{\prime
\prime }(x)\right\vert <\infty $ in (\ref{2.2})$,$ then the inequality in (%
\ref{2.1}) is better than the inequality in (\ref{2.2}).

\textbf{Proof of Theorem \ref{teo 2.3}: }From Lemma \ref{lem 1.1} with
properties of modulus we get

\begin{eqnarray}
&&\left\vert \frac{f\left( a\right) +f\left( b\right) }{2}-\frac{1}{b-a}%
\int_{a}^{b}f\left( x\right) dx\right\vert  \label{3.4} \\
&\leq &\frac{\left( b-a\right) ^{2}}{2}\int_{0}^{1}t\left( 1-t\right)
\left\vert f^{\prime \prime }\left( ta+\left( 1-t\right) b\right)
\right\vert dt.  \notag
\end{eqnarray}

Now, if we use the following weighted version of H\"{o}lder's inequality 
\cite[p. 117]{5}:%
\begin{equation}
\left\vert \int_{I}f(s)g(s)h(s)ds\right\vert \leq \left( \int_{I}\left\vert
f(s)\right\vert ^{p}h(s)ds\right) ^{\frac{1}{p}}\left( \int_{I}\left\vert
g(s)\right\vert ^{q}h(s)ds\right) ^{\frac{1}{q}}  \label{3.5}
\end{equation}%
for $p>1,p^{-1}+q^{-1}=1,$ $h$ is nonnegative on $I$ and provided all the
other integrals exist and are finite.

If we rewrite the inequality (\ref{3.4}) with respect to (\ref{3.5}) with $%
\left\vert f^{\prime \prime }\right\vert ^{q}$ is $quasi-$convex on $[a,b]$
for all $t\in \lbrack 0,1],$ we get%
\begin{eqnarray*}
&&\left\vert \frac{f\left( a\right) +f\left( b\right) }{2}-\frac{1}{b-a}%
\int_{a}^{b}f\left( x\right) dx\right\vert \\
&\leq &\frac{\left( b-a\right) ^{2}}{2}\int_{0}^{1}t\left( 1-t\right)
\left\vert f^{\prime \prime }\left( ta+\left( 1-t\right) b\right)
\right\vert dt \\
&=&\frac{\left( b-a\right) ^{2}}{2}\int_{0}^{1}\left( 1-t\right) \left\vert
f^{\prime \prime }\left( ta+\left( 1-t\right) b\right) \right\vert tdt \\
&\leq &\frac{\left( b-a\right) ^{2}}{2}\left( \int_{0}^{1}\left( 1-t\right)
^{p}tdt\right) ^{\frac{1}{p}}\left( \int_{0}^{1}\left\vert f^{\prime \prime
}\left( ta+\left( 1-t\right) b\right) \right\vert ^{q}tdt\right) ^{\frac{1}{q%
}} \\
&=&\frac{\left( b-a\right) ^{2}}{2}\left( \beta \left( 2,p+1\right) \right)
^{\frac{1}{p}}\left( \frac{\max \left\{ \left\vert f^{\prime \prime
}(a)\right\vert ^{q},\left\vert f^{\prime \prime }(b)\right\vert
^{q}\right\} }{2}\right) ^{\frac{1}{q}}.
\end{eqnarray*}%
The proof of Theorem \ref{teo 2.3} is completed.

\begin{corollary}
\label{co 2.3} In Theorem \ref{teo 2.3}, if we choose $M=Sup_{x\in
(a,b)}\left\vert f^{\prime \prime }(x)\right\vert <\infty ,$ we get%
\begin{equation*}
\left\vert \frac{f\left( a\right) +f\left( b\right) }{2}-\frac{1}{b-a}%
\int_{a}^{b}f\left( x\right) dx\right\vert \leq \frac{\left( b-a\right) ^{2}%
}{2^{1+\frac{1}{q}}}M\left( \beta \left( 2,p+1\right) \right) ^{\frac{1}{p}}.
\end{equation*}
\end{corollary}

\begin{remark}
\label{rem 2.1} From Theorems \ref{teo 2.1}-\ref{teo 2.3}, we get%
\begin{equation*}
\left\vert \frac{f\left( a\right) +f\left( b\right) }{2}-\frac{1}{b-a}%
\int_{a}^{b}f\left( x\right) dx\right\vert \leq \min \left\{
v_{1},v_{2},v_{3}\right\}
\end{equation*}%
where 
\begin{equation*}
v_{1}=\frac{\left( b-a\right) ^{2}}{2}\left( \frac{q-1}{2q-p-1}\right) ^{%
\frac{q-1}{q}}\left( \beta \left( p+1,q+1\right) \right) ^{\frac{1}{q}%
}\times \left( \max \left\{ \left\vert f^{\prime \prime }(a)\right\vert
^{q},\left\vert f^{\prime \prime }(b)\right\vert ^{q}\right\} \right) ^{%
\frac{1}{q}},
\end{equation*}%
\begin{equation*}
v_{2}=\frac{\left( b-a\right) ^{2}}{4}\left( \frac{2}{\left( q+1\right)
\left( q+2\right) }\right) ^{\frac{q-1}{q}}\left( \max \left\{ \left\vert
f^{\prime \prime }(a)\right\vert ^{q},\left\vert f^{\prime \prime
}(b)\right\vert ^{q}\right\} \right) ^{\frac{1}{q}}
\end{equation*}%
and%
\begin{equation*}
v_{3}=\frac{\left( b-a\right) ^{2}}{2^{1+\frac{1}{q}}}\left( \beta \left(
2,p+1\right) \right) ^{\frac{1}{p}}\left( \max \left\{ \left\vert f^{\prime
\prime }(a)\right\vert ^{q},\left\vert f^{\prime \prime }(b)\right\vert
^{q}\right\} \right) ^{\frac{1}{q}}.
\end{equation*}
\end{remark}

\section{Error Estimates for the Trapezoidal Rule}

Let $d$ be a partition $a=x_{0}<x_{1}<x_{2}<...<x_{n}=b$ of the interval $%
[a,b]$ and consider the quadrature formula%
\begin{equation}
\int_{a}^{b}f(x)dx=T_{i}(f,d)+E_{i}(f,d),\text{ \ \ \ }i=1,2,...,n-1
\label{4.1}
\end{equation}%
where 
\begin{equation*}
T_{1}(f,d)=\sum_{i=0}^{n-1}\frac{f(x_{i})+f(x_{i+1})}{2}\left(
x_{i+1}-x_{i}\right)
\end{equation*}%
for the Trapezoidal version and 
\begin{equation*}
T_{2}(f,d)=\sum_{i=0}^{n-1}f\left( \frac{x_{i}+x_{i+1}}{2}\right) \left(
x_{i+1}-x_{i}\right)
\end{equation*}%
for the Midpoint formula and $E_{i}(f,d)$ denotes the associated
approximation errors.

\begin{proposition}
\label{proposition 1} Suppose that all the assumptions of Theorem \ref{teo
2.1} are satisfied for every division $d$ of $[a,b],$ we have 
\begin{eqnarray*}
\left\vert E(f,d)\right\vert  &\leq &\frac{1}{2}\left( \frac{q-1}{2q-p-1}%
\right) ^{\frac{q-1}{q}}\left( \beta \left( p+1,q+1\right) \right) ^{\frac{1%
}{q}} \\
&&\times \sum_{i=0}^{n-1}\left( x_{i+1}-x_{i}\right) ^{3}\left( \max \left\{
\left\vert f^{\prime \prime }(x_{i})\right\vert ^{q},\left\vert f^{\prime
\prime }(x_{i+1})\right\vert ^{q}\right\} \right) ^{\frac{1}{q}}.
\end{eqnarray*}
\end{proposition}

\begin{proof}
Appliying Theorem \ref{teo 2.1} on the subinterval $\left(
x_{i+1},x_{i}\right) ,i=1,2,...,n-1$ of the partition and by using the $%
quasi-$convexity of $\left\vert f^{\prime \prime }\right\vert ^{q},$ we
obtain%
\begin{eqnarray*}
&&\left\vert \frac{f(x_{i})+f(x_{i+1})}{2}-\frac{1}{\left(
x_{i+1}-x_{i}\right) }\int_{x_{i}}^{x_{i+1}}f(x)dx\right\vert  \\
&\leq &\frac{\left( x_{i+1}-x_{i}\right) ^{2}}{2}\left( \frac{q-1}{2q-p-1}%
\right) ^{\frac{q-1}{q}}\left( \beta \left( p+1,q+1\right) \right) ^{\frac{1%
}{q}} \\
&&\times \left( \max \left\{ \left\vert f^{\prime \prime }(x_{i})\right\vert
^{q},\left\vert f^{\prime \prime }(x_{i+1})\right\vert ^{q}\right\} \right)
^{\frac{1}{q}}.
\end{eqnarray*}%
Hence in (\ref{4.1}), we have 
\begin{eqnarray*}
\left\vert \int_{a}^{b}f(x)dx-T(f,d)\right\vert  &=&\left\vert
\sum_{i=0}^{n-1}\left\{ \int_{x_{i}}^{x_{i+1}}f(x)dx-\frac{%
f(x_{i})+f(x_{i+1})}{2}\left( x_{i+1}-x_{i}\right) \right\} \right\vert  \\
&\leq &\sum_{i=0}^{n-1}\left\vert \int_{x_{i}}^{x_{i+1}}f(x)dx-\frac{%
f(x_{i})+f(x_{i+1})}{2}\left( x_{i+1}-x_{i}\right) \right\vert  \\
&\leq &\frac{1}{2}\left( \frac{q-1}{2q-p-1}\right) ^{\frac{q-1}{q}}\left(
\beta \left( p+1,q+1\right) \right) ^{\frac{1}{q}} \\
&&\times \sum_{i=0}^{n-1}\left( x_{i+1}-x_{i}\right) ^{3}\left( \max \left\{
\left\vert f^{\prime \prime }(x_{i})\right\vert ^{q},\left\vert f^{\prime
\prime }(x_{i+1})\right\vert ^{q}\right\} \right) ^{\frac{1}{q}}.
\end{eqnarray*}
\end{proof}

\begin{proposition}
\label{proposition 3} Suppose that all the assumptions of Theorem \ref{teo
2.2} are satisfied for every division $d$ of $[a,b],$ we have%
\begin{eqnarray*}
\left\vert E(f,d)\right\vert &\leq &\frac{1}{4}\left( \frac{2}{(q+1)(q+2)}%
\right) ^{\frac{1}{q}} \\
&&\times \sum_{i=0}^{n-1}\left( x_{i+1}-x_{i}\right) ^{3}\left( \max \left\{
\left\vert f^{\prime \prime }(x_{i})\right\vert ^{q},\left\vert f^{\prime
\prime }(x_{i+1})\right\vert ^{q}\right\} \right) ^{\frac{1}{q}}.
\end{eqnarray*}
\end{proposition}

\begin{proof}
The proof is immediate follows from Theorem \ref{teo 2.2} and by applying a
similar argument to the Proposition \ref{proposition 1}.
\end{proof}

\begin{proposition}
\label{proposition 4} Suppose that all the assumptions of Theorem \ref{teo
2.3} are satisfied for every division $d$ of $[a,b],$ we have%
\begin{eqnarray*}
\left\vert E(f,d)\right\vert &\leq &\frac{1}{2}\left( \beta \left(
1,p+1\right) \right) ^{\frac{1}{p}} \\
&&\times \sum_{i=0}^{n-1}\left( x_{i+1}-x_{i}\right) ^{2}\left( \frac{\max
\left\{ \left\vert f^{\prime \prime }(a)\right\vert ^{q},\left\vert
f^{\prime \prime }(b)\right\vert ^{q}\right\} }{2}\right) ^{\frac{1}{q}}.
\end{eqnarray*}
\end{proposition}

\begin{proof}
The proof is immediate follows from Theorem \ref{teo 2.3} and by applying a
similar argument to the Proposition \ref{proposition 1}.
\end{proof}

\section{Applications to special means}

Let us consider the special means for real numbers $a,b$ $(a\neq b).$ We take

\textbf{1. Arithmetic mean:}%
\begin{equation*}
A(a,b)=\frac{a+b}{2},\text{ \ \ }a,b\in 
\mathbb{R}
.
\end{equation*}

\textbf{2. Logarithmic mean: }%
\begin{equation*}
L(a,b)=\frac{a-b}{\ln \left\vert a\right\vert -\ln \left\vert b\right\vert },%
\text{ \ \ }\left\vert a\right\vert \neq \left\vert b\right\vert ,a,b\neq
0,a,b\in 
\mathbb{R}
.
\end{equation*}

\textbf{3. Generalized log-mean:}%
\begin{equation*}
L_{n}(a,b)=\left[ \frac{b^{n+1}-a^{n+1}}{(n+1)(b-a)}\right] ^{\frac{1}{n}},%
\text{ \ \ }n\in 
\mathbb{Z}
\backslash \left\{ -1,0\right\} ,a,b\in 
\mathbb{R}
,a\neq b.
\end{equation*}

\begin{proposition}
\label{proposition 5} Let $a,b\in 
\mathbb{R}
,$ $a<b$ and $n\in 
\mathbb{N}
,$ $n\geq 2.$ Then we have%
\begin{eqnarray*}
&&\left\vert A(a^{n},b^{n})-L_{n}^{n}(a,b)\right\vert \leq \frac{%
n(n-1)\left( b-a\right) ^{2}}{2} \\
&&\times \left( \frac{q-1}{2q-p-1}\right) ^{\frac{q-1}{q}}\left( \beta
\left( p+1,q+1\right) \right) ^{\frac{1}{q}}\left( \max \left\{ \left\vert
a\right\vert ^{(n-2)q},\left\vert b\right\vert ^{(n-2)q}\right\} \right) ^{%
\frac{1}{q}}.
\end{eqnarray*}
\end{proposition}

\begin{proof}
The assertion follows from Theorem \ref{teo 2.1} applied to the $quasi-$%
convex mapping $f(x)=x^{n},$ $x\in 
\mathbb{R}
.$
\end{proof}

\begin{proposition}
\label{proposition 6} Let $a,b\in 
\mathbb{R}
,$ $a<b$ and $n\in 
\mathbb{N}
,$ $n\geq 2.$ Then we have%
\begin{eqnarray*}
&&\left\vert A(a^{n},b^{n})-L_{n}^{n}(a,b)\right\vert \leq \frac{%
n(n-1)\left( b-a\right) ^{2}}{4} \\
&&\times \left( \frac{2}{(q+1)(q+2)}\right) ^{\frac{1}{q}}\left( \max
\left\{ \left\vert a\right\vert ^{(n-2)q},\left\vert b\right\vert
^{(n-2)q}\right\} \right) ^{\frac{1}{q}}.
\end{eqnarray*}
\end{proposition}

\begin{proof}
The assertion follows from Theorem \ref{teo 2.2} applied to the $quasi-$%
convex mapping $f(x)=x^{n},$ $x\in 
\mathbb{R}
.$
\end{proof}

\begin{proposition}
\label{proposition 7} Let $a,b\in 
\mathbb{R}
,$ $a<b$ and $n\in 
\mathbb{N}
,$ $n\geq 2.$ Then we have%
\begin{eqnarray*}
&&\left\vert A(a^{n},b^{n})-L_{n}^{n}(a,b)\right\vert \leq \frac{%
n(n-1)\left( b-a\right) ^{2}}{2^{1+\frac{1}{q}}} \\
&&\times \left( \beta \left( 2,q+1\right) \right) ^{\frac{1}{p}}\left( \max
\left\{ \left\vert a\right\vert ^{(n-2)q},\left\vert b\right\vert
^{(n-2)q}\right\} \right) ^{\frac{1}{q}}.
\end{eqnarray*}
\end{proposition}

\begin{proof}
The assertion follows from Theorem \ref{teo 2.3} applied to the $quasi-$%
convex mapping $f(x)=x^{n},$ $x\in 
\mathbb{R}
.$
\end{proof}

\end{document}